\documentclass[12pt,final]{amsart}
 
\usepackage{tikz} 
\usepackage{a4wide}
\usepackage[notcite,notref,color]{showkeys} 
\usepackage{amsmath,amssymb,amsthm,bbm} 
\usepackage{enumitem} 
\usepackage[overload]{textcase}
 
\usepackage{epsfig,psfrag,color,float,graphics} 
\usepackage{verbatim}

\usepackage[colorlinks,bookmarks,linkcolor=black,citecolor=black]{hyperref} 

\usepackage{amsthm}

\theoremstyle{plain}

\setlist{itemsep=3pt,parsep=0pt,topsep=2pt,partopsep=0pt}  
\setlist{leftmargin=2.5\parindent} 

\def\endofClaim{\hfill\scalebox{.6}{$\Box$}}

\let\subset\subseteq  
 
\let\rho\varrho 
\def\leq{\leqslant}
\def\geq{\geqslant}

 \newtheorem*{theorem*}{Theorem}

\newtheorem{theorem}{Theorem}%[section] 

\newtheorem{lemma}[theorem] {Lemma}    
    
\newtheorem{prop}[theorem] {Proposition}

\newtheorem{claim}[theorem]{Claim} 
\theoremstyle{definition}

\theoremstyle{remark}

\newcommand{\oldqed}{}

\newenvironment{claimproof}[1][Proof]{
  \renewcommand{\oldqed}{\qedsymbol}
  \renewcommand{\qedsymbol}{\endofClaim}
  \begin{proof}[#1]
}{
  \end{proof}
  \renewcommand{\qedsymbol}{\oldqed}
}
\def\NN{\mathbb{N}}

\newcommand{\sat}{\mathrm{sat}}
\newcommand{\satp}{\mathrm{sat_p}}
\newcommand{\exsatp}{\mathrm{exsat_p}}
\newcommand{\vbad}{\mathrm{V_{bad}}}

\usetikzlibrary{calc, shapes, backgrounds}
\usetikzlibrary{arrows,decorations.pathmorphing,positioning,fit,petri}

\title{Partite Saturation Problems}
\author{Barnaby Roberts}
\email{b.j.roberts@lse.ac.uk}
\address{Department of Mathematics, London School of Economics, Houghton Street, London
WC2A 2AE, U.K.}

\begin{document}
\begin{abstract}
We look at several saturation problems in complete balanced blow-ups of graphs.  We let $H[n]$ denote the blow-up of $H$ onto parts of size $n$ and refer to a copy of $H$ in $H[n]$ as \emph{partite} if it has one vertex in each part of $H[n]$.  We then ask how few edges a subgraph $G$ of $H[n]$ can have such that $G$ has no partite copy of $H$ but such that the addition of any new edge from $H[n]$ creates a partite $H$.  
When $H$ is a triangle this value was determined by Ferrara, Jacobson, Pfender, and Wenger in~\cite{FerJacPfeWen}.  Our main result is to calculate this value for $H=K_4$ when $n$ is large.  We also give exact results for paths and stars and show that for $2$-connected graphs the answer is linear in $n$ whilst for graphs which are not $2$-connected the answer is quadratic in $n$.
We also investigate a similar problem where $G$ is permitted to contain partite copies of $H$ but we require that the addition of any new edge from $H[n]$ creates an extra partite copy of $H$.  This problem turns out to be much simpler and we attain exact answers for all cliques and trees.
\end{abstract}

\maketitle

\section{Introduction}
The Tur\'an problem of asking for the maximum number of edges a graph on a fixed number of vertices can have without containing some fixed subgraph $H$ is one of the oldest and most famous questions in extremal graph theory, see \cite{Mantel},\cite{Turan},\cite{ErdStone}.

Since the corresponding minimisation problem - asking how few edges an $H$-free graph can have - trivially gives the answer zero, if we want an interesting complementary question to the Tur\'an problem we can require that our $H$-free graph $G$ also has the property that it \emph{nearly} contains a copy of $H$.
By this we mean that the addition of any new edge to $G$ creates an copy of $H$ as a subgraph.
Such a graph $G$ is called \emph{$H$-saturated} and over $H$-saturated graphs on $n$ vertices the minimum number of edges is called the saturation number, $\sat(H,n)$.
The study of saturation numbers was initiated by Erd\H{o}s, Hajnal and Moon~\cite{ErdHajMoon} when they proved that $\sat(K_r,n)=(r-2)(n-\frac{1}{2}(r-1))$.  
It was later shown by K\'aszonyi and Tuza in~\cite{KasTuza} that cliques have the largest saturation number of any graph on $r$ vertices which in particular implies that for any $H$ the saturation number $\sat(H,n)$ grows linearly in $n$.

These saturation questions can be generalised to require our $H$-free graph $G$ to be a subgraph of another fixed graph $F$.
Here we insist that adding any new edge of $F$ to $G$ would create a copy of $H$ in $G$.
The minimum number of edges in such a $G$ we denote by $\sat(H,F)$.
One natural class of host graphs are complete $r$-partite graphs.
In the bipartite case Bollob\'as~\cite{Bol-ExWeights,Bol-EHM} and Wessel~\cite{Wessel1, Wessel2} independently determined the saturation number $\sat(K_{a,b},K_{c,d})$.
Working in the $r$-partite setting with $r \geq 3$, Ferrara, Jacobson, Pfender, and Wenger determined in~\cite{FerJacPfeWen} the value of $\sat(K_3,K_r^n)$ for sufficiently large $n$ and showed that $\sat(K_3,K_3^n)=6n-6$ for all $n$.

In this paper we consider the saturation problem when the host graph is a blow-up of the forbidden subgraph $H$.
For any graph $H$ and any $n \in \NN$ let \emph{$H[n]$} denote the graph obtained from $H$ by replacing each vertex with an independent set of size $n$ and each edge with a complete bipartite graph between the corresponding independent sets.  
A copy of $H$ in $H[n]$ is called $\emph{partite}$ if it has exactly one vertex in each part of $H[n]$.  
For a subgraph $G$ of $H[n]$ we say $G$ is \emph{$H$-partite-free} if there is no partite copy of $H$ in $G$.  
We say $G$ is \emph{$(H,H[n])$-partite-saturated} if $G$ is $H$-partite-free but for any $uv \in E(H[n]\setminus G)$ the graph $G\cup uv$ is not $H$-partite-free.  
We consider the problem of determining the value
\begin{equation*}
\satp(H,H[n]):=\min\big\{e(G):G\subset H[n] \text{ is $(H,H[n])$-partite-saturated}\big\}
\end{equation*}
for graphs $H$.

Note that for a graph $H$ with no homomorphism onto any proper subgraph of itself we have by definition $\satp(H,H[n])=\sat(H,H[n])$.
In this way we know that $\satp(K_3,K_3[n])=6n-6$ from~\cite{FerJacPfeWen} and can drop the partite requirement when considering cliques.
Our main result, Theorem~\ref{Thm:K4Sat}, is to show that for sufficiently large $n$ we have $\sat(K_4,K_4[n])=18n-21$.
In addition we calculate the partite-saturation numbers of stars and paths in Theorems~\ref{Thm:StarSat} and~\ref{Thm:PathSat} respectively.

In the original paper by Erd\H{o}s, Hajnal and Moon they did not in fact require the graph $G$ to be $H$-free but only required that the addition of any edge would create an extra copy of $H$.
Interestingly for the problem they studied this did not have an effect as the extremal graphs were $K_r$-free even without requiring this restriction.
We consider a similar notion in the partite setting.
For $G\subset H[n]$ and $n\in \NN$ we say $G$ is \emph{$(H,H[n])$-partite-extra-saturated} if for any $uv \in E(H[n]\setminus G)$ the graph $G\cup uv$ has more partite copies of $H$ than $G$.
We also ask, given a graph $H$ and $n\in \NN$, the value of
\begin{equation*}
\exsatp(H,H[n]):=\min\big\{e(G):G\subset H[n] \text{ is $(H,H[n])$-partite-extra-saturated}\big\}\,.
\end{equation*}

We observe some interesting differences in behaviour between these partite saturation numbers and the saturation numbers studied by Erd\H{o}s, Hajnal and Moon.
Whilst for graphs on $r$ vertices cliques gave the largest values of $\sat(H,n)$ we find that cliques are not the graphs which maximise $\satp(H,H[n])$.  
In fact we find in Theorem~\ref{Thm:2-connectivity} that $\satp(H,H[n])$ grows quadratically for graphs $H$ which are not $2$-connected whilst it grows linearly for those which are.
On the other-hand we show in Theorem~\ref{Thm:CliqueExSat} that cliques do maximise the partite-extra-saturation numbers and that all partite-extra-saturation numbers are linear.

\smallskip

\paragraph{\bf Notation}
Most of the notation we use is standard.
In a graph $G$ for a vertex $v \in V(G)$ and a set $X \subset V(G)$ we let $\deg_G(v,X)$ denote the number of neighbours of $v$ in $X$.
Where $X=V(G)$ we will abbreviate to $\deg_G(v)$ and when the graph $G$ is clear from the context we will omit the subscript.
For a vertex $v$ we let $N(v)$ denote the set of neighbours of $v$.

\paragraph{\bf Organisation}

Section~\ref{Section:K4} is dedicated to determining the partite saturation number of $K_4$.
In Section~\ref{Section:PathsStars} we then determine the partite saturation numbers of paths and stars.
We look at the link between $2$-connectivity and the order of magnitude of partite saturation numbers in Section~\ref{Section:2-connectivity} before focusing on partite extra-saturation numbers in Section~\ref{Section:ExSat}.
Finally in Section~\ref{Section:Conclusion} we give some further remarks and open problems.

\section{The Partite Saturation Number of $K_4$}\label{Section:K4}

\begin{theorem}\label{Thm:K4Sat}
For all large enough $n \in \NN$ we have
\begin{equation*}
\sat(K_4,K_4[n])= 18n-21\,.
\end{equation*}
\end{theorem}
Furthermore we determine the unique graph for which equality holds.

We first give a construction of a graph $G \subset K_4[n]$ that is $(K_4,K_4[n])$-saturated and has $18n-21$ edges.

Let $X_1,X_2,X_3,X_4$ be the parts of $K_4[n]$.
Choose vertices $x_i$ and $x_i'$ in each $X_i$.
Let $Z$ denote the set of these 8 vertices.
Include in $G$ the following 15 edges $x_1x_2$, $x_1x_2'$, $x_1x_3'$, $x_1x_4'$, $x_1'x_2'$, $x_1'x_3$, $x_1'x_4$, $x_2x_3$, $x_2x_4$, $x_2x_4'$, $x_2'x_3'$, $x_2'x_4$, $x_3x_4'$, $x_3'x_4$, $x_3'x_4'$.
We now only add edges between $Z$ and $V(G)\setminus Z$.
Include all edges between $X_1 \setminus Z$ and each of $x_2$, $x_3$, $x_3'$ and $x_4$.
Attach all vertices in $X_2\setminus Z$ to $x_1'$, $x_3$, $x_3'$, $x_4$ and $x_4'$.
Join all of $X_3\setminus Z$ to each of $x_1$, $x_1'$, $x_2$ and $x_4$ and finally add all edges from $X_4\setminus Z$ to $x_1$, $x_1'$, $x_2$, $x_2'$ and $x_3$.

\begin{center}
\begin{figure}[H]

\begin{tikzpicture}
\node (x_1) at ( 2.5,7) [circle, fill=black,inner sep=2pt,label=left:$x_1$] {};
\node (x_1') at ( 4.5,9) [circle, fill=black,inner sep=2pt,label=above:$x_1'$] {};
\node (x_2) at ( 2.5,3.5) [circle, fill=black,inner sep=2pt,label=left:$x_2$] {};
\node (x_2') at ( 4.5,1.5) [circle, fill=black,inner sep=2pt,label=below:$x_2'$] {};
\node (x_3) at ( 10,3.5) [circle, fill=black,inner sep=2pt,label=right:$x_3$] {};
\node (x_3') at ( 8,1.5) [circle, fill=black,inner sep=2pt,label=below:$x_3'$] {};
\node (x_4) at ( 8,9) [circle, fill=black,inner sep=2pt,label=above:$x_4$] {};
\node (x_4') at ( 10,7) [circle, fill=black,inner sep=2pt,label=right:$x_4'$] {};

\node (X_1) at (1.75,8.5) [rectangle,rounded corners=2mm,draw,fill=white,align=center] {$X_1 \setminus Z$\\ all vertices adjacent\\ to: $x_2$, $x_3$, $x_3'$, $x_4$};
\node (X_2) at (1.75,2) [rectangle,rounded corners=2mm,draw,fill=white,align=center] {$X_2 \setminus Z$\\ all vertices adjacent\\ to: $x_1'$, $x_3$, $x_3'$, $x_4$, $x_4'$};
\node (X_3) at (10.75,2) [rectangle,rounded corners=2mm,draw,fill=white,align=center] {$X_3 \setminus Z$\\ all vertices adjacent\\ to: $x_1$, $x_1'$, $x_2$, $x_4$};
\node (X_4) at (10.75,8.5) [rectangle,rounded corners=2mm,draw,fill=white,align=center] {$X_4 \setminus Z$\\ all vertices adjacent\\ to: $x_1$, $x_1'$, $x_2$, $x_2'$, $x_3$};

 \foreach \from/\to in {x_1/x_2',x_1/x_3',x_1/x_4',x_1'/x_2',x_2/x_4',x_2'/x_3',x_2'/x_4,x_3/x_4',x_3'/x_4'}
    \draw [blue] (\from) -- (\to);
    
\foreach \from/\to in {x_1/x_2,x_1'/x_3,x_1'/x_4,x_2/x_3,x_2/x_4,x_3'/x_4}
    \draw [blue] (\from) -- (\to);
\end{tikzpicture}

\caption{$K_4$-Partite-Saturation Construction}
\end{figure}
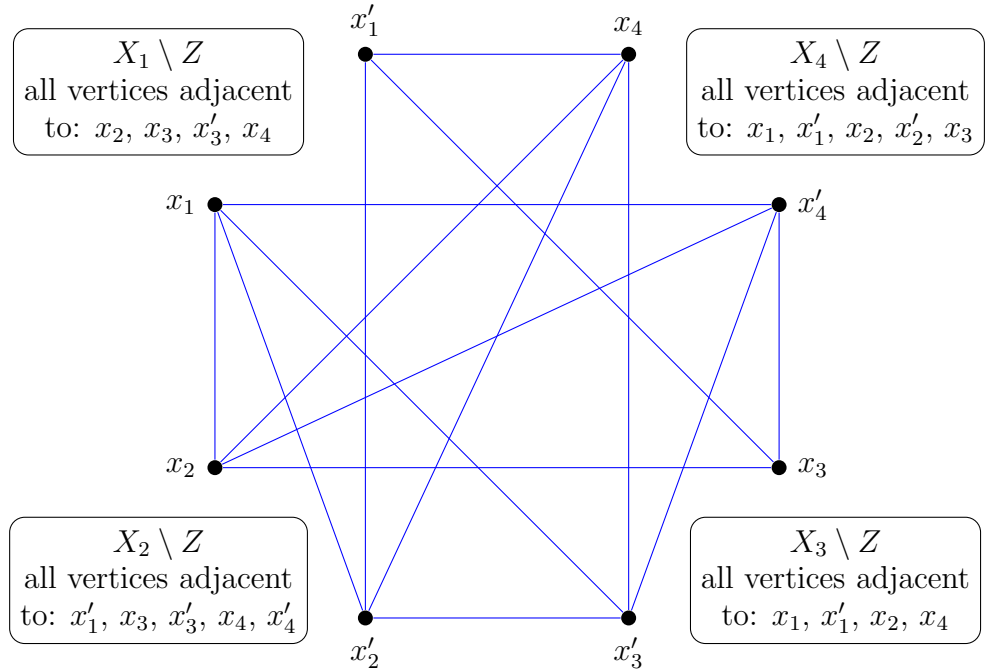
\end{center}

\begin{prop}\label{Cl:Construction}
$G$ is a $(K_4,K_4[n])$-saturated graph with $18n-21$ edges.
\end{prop}
\begin{proof}
To see that this graph is $K_4$-free note that the graph induced on $V(G)\setminus Z$ has no edges so any $K_4$ would have to come from a triangle in $Z$ extended to a vertex outside of $Z$.
There are just six triangles induced on $Z$ and none of them extend to a $K_4$.

To see that $G$ is $(K_4,K_4[n])$-saturated we first observe that for any pair $i,j$ there is an edge in $Z$ such that $(X_i \cup X_j) \setminus Z$ is contained in the common neighbourhood of the ends of that edge.
Therefore we could only add an edge with at least one end in $Z$.

For a vertex $v \in X_1\setminus Z$ the only incident edges we could add are $vx_2'$ or $vx_4'$.
These additional edges would create a $K_4$ on $vx_2'x_3'x_4$ or $vx_2x_3x_4'$ respectively.
For a vertex $v \in X_2 \setminus Z$ the only incident edge we could add is $vx_1$ but this would create a $K_4$ on $x_1vx_3'x_4'$.
Similar arguments show we cannot add edges incident to $X_3 \setminus Z$ and $X_4 \setminus Z$.
Adding any edge to $Z$ that has either $x_1$ or $x_3'$ as an endpoint will create a $K_4$ in $Z$.
Adding any other edge of $Z$ will create a triangle on $Z$ that extends to a $K_4$ with a vertex outside of $Z$.
That $G$ has $18n-21$ edges is easy to check.
\end{proof}

\medskip

Before proving a matching lower bound we need the following lemmas.
\begin{lemma}\label{Lem:MinDeg4}
Any $(K_4,K_4[n])$-saturated graph $G$ with $n \geq 2$ has minimum degree at least $4$.
\end{lemma}
\begin{proof}
Let $G$ be a $(K_4,K_4[n])$-saturated graph on $X_1 \cup \cdots \cup X_4$.
Suppose for contradiction that there exists $a_1 \in X_1$ with at most 3 neighbours.  
If $a_1$ has no neighbours in one part, say $X_2$, then by saturation it must be adjacent to all vertices in the other parts, which for $n \geq 2$ contradicts the fact that $\deg(a_1)\leq 3$.  
So $a_1$ must have exactly three neighbours with one in each of the parts.  
Call these  $x_i \in X_i$ for $i=2,3,4$.  
Then for any $i=2,3,4$ adding the edge $a_1y_i$ for some $y_i \in X_i \setminus x_i$ must create a $K_4$.  
This implies that $x_2x_3$, $x_2x_4$ and $x_3x_4$ are all edges of $G$ but along with $a_1$ this gives a $K_4$.
\end{proof}

We can also say more about the neighbourhoods of vertices with degree exactly $4$.
\begin{lemma}\label{Lem:Deg4Vertices}
Let $G$ be a $(K_4,K_4[n])$-saturated graph on $X_1 \cup \cdots \cup X_4$ with $n \geq 3$ and let $v$ be a vertex of degree exactly $4$.
Then $v$ has one neighbour in each of two parts and two neighbours in one part.
The neighbourhood of $v$ induces a path beginning and ending with the vertices in the same part.
All neighbours of $v$ have degree at least $n-2$.
\end{lemma}
\begin{proof}
Suppose $v \in X_1$.
If $v$ had no neighbour in some $X_i$ ($i \neq 1$) it would be adjacent to all vertices in other parts meaning it would have degree greater than $4$.
Suppose without loss of generality that the neighbours of $v$ are $x_2$, $x_3$, $x_3'$ and $x_4$ with the subscripts denoting the parts containing each vertex.
By considering the effect of adding the edge $vy_3$ for some $y_3 \in X_3 \setminus \{x_3,x_3'\}$ we see that the edge $x_2x_4$ is present.
We also see that all vertices in $X_3 \setminus \{x_3,x_3'\}$ are adjacent to $x_2$ and $x_4$.
Similarly by considering a vertex in $X_2 \setminus \{x_2\}$ we see that there must be an edge between $x_4$ and one of $x_3$ or $x_3'$.
Without loss of generality assume $x_4x_3'$ is present.
Finally by considering a vertex in $X_4 \setminus \{x_4\}$ we see that $x_2$ is adjacent to either $x_3$ or $x_3'$.
In order not to create a $K_4$ it must be that $x_2x_3$ is present.
We now cannot have the edges $x_4x_3$ or $x_2x_3'$.
We then see that all vertices in $X_4 \setminus \{x_4\}$ are adjacent to $x_3$ and all vertices in $X_2 \setminus \{x_2\}$ are adjacent to $x_3'$.
Hence the neighbours of $v$ all have degree at least $n-2$.
\end{proof}
It follows that when $n >6$ vertices of degree exactly $4$ cannot be adjacent.

The following lemma gives us minimum degree conditions that more reflect those of the upper bound construction.
\begin{lemma}\label{Lem:Deg4Parts}
Let $G$ be a $(K_4,K_4[n])$-saturated graph with $n \geq 22$ on $X_1 \cup \cdots \cup X_4$.
There cannot be two degree $4$ vertices, $a_i \in X_i$ and $a_j \in X_j$ with $i \neq j$ such that $a_i$ has just one neighbour in $X_j$.
Furthermore there are at most two parts with minimum degree $4$.
\end{lemma}
\begin{proof}
Suppose for contradiction that $a_1 \in X_1$ and $a_2 \in X_2$ are degree $4$ vertices such that $a_1$ has just one neighbour in $X_2$ and let $x_2,x_3,x_3',x_4$ denote the neighbours of $a_1$.  
Then (up to switching between $x_3'$ and $x_3$) the edges $x_2x_3,x_2x_4,x_3'x_4$ are all present.  
We also know that $x_2$ is adjacent to all of $(X_3 \cup X_4) \setminus x_3'$, that $x_3$ is adjacent to all of $X_4 \setminus x_4$, that $x_3'$ is adjacent to all of $X_2 \setminus x_2$, and $x_4$ is adjacent to all of $(X_2 \cup X_3)\setminus x_3$.
In particular this implies we have the edges $a_2x_3'$ and $a_2x_4$.  
The vertex $a_2$ also has some neighbour $x_1 \in X_1 \setminus a_1$.  
As $a_2$ has degree $4$ it must have one more neighbour.  
We split into cases depending on where this final neighbour is and show that each case leads to a contradiction.
The possible cases are:
\begin{enumerate}[label=(\roman*)]
\item $a_2$ has another neighbour $v \in (X_1 \cup X_3) \setminus \{a_1,x_1,x_3,x_3'\}$.

\item $a_2$ is adjacent to $x_3$.

\item $a_2$ has another neighbour $x_4' \in X_4 \setminus x_4$.
\end{enumerate}

Case i) Since $x_3$ is not adjacent to $a_2$ it must be adjacent to $x_4$ as $X_4 \cap N(y_2)=\{x_4\}$ and hence $x_1x_2x_3x_4$ forms a $K_4$.
\smallskip

Case ii) By considering vertices in $X_3 \setminus N(a_2)$ we must have the edge $x_1x_4$ and we see that $x_1$ is adjacent to all of $X_3 \setminus \{x_3,x_3'\}$. 
We also see that all vertices in $X_1 \setminus N(a_2)$ are adjacent to $x_3'$ and $x_4$.
This means that in fact all vertices in $(X_1 \cup X_2)\setminus \{x_1,x_2\}$ are adjacent to $x_3'$ and $x_4$ and hence all edges in $X_1 \cup X_2$ have one end in $\{x_1,x_2\}$.
In fact all edges in $X_1 \cup X_2$ have exactly one end in $\{x_1,x_2\}$ as if the edge $x_1x_2$ were present this would create a $K_4$ with $x_4$ and any vertex in $X_3 \setminus \{x_3,x_3'\}$.

If all vertices in $X_3 \setminus \{x_3,x_3'\}$ were adjacent to all of $X_4 \setminus x_4$ this would give at least $(n-2)(n-1)$ edges which is greater than $18n$ for $n \geq 22$.
Therefore consider some vertex $v_3 \in X_3 \setminus \{x_3,x_3'\}$ which is non-adjacent to some $v_4 \in X_4 \setminus x_4$.
As $v_4$ is non-adjacent to $v_3$ it must be adjacent to both ends of an edge in $N(v_3)\cap (X_1 \cup X_2)$.
We know that this edge has exactly one end in $\{x_1,x_2\}$ but this creates a $K_4$ with $v_3$ and $x_4$.
\smallskip

Case iii) As $x_4'$ is not adjacent to $a_1$ it is adjacent to $x_2$ and $x_3$.  
By considering vertices in $X_4 \setminus N(a_2)$ we see that $x_1x_3'$ is an edge of $G$ and all vertices in $X_4 \setminus N(a_2)$ are adjacent to $x_1$ and $x_3'$.  
By considering vertices in $X_3 \setminus N(a_2)$ we see that $x_1x_4'$ is an edge of $G$ (as $x_1x_4$ would create a $K_4$) and all vertices in $X_3 \setminus N(a_2)$ are adjacent to $x_1$ and $x_4'$.
Finally by considering vertices in $X_1 \setminus N(a_2)$ we observe that all vertices in $X_1\setminus x_1$ are adjacent to $x_3'$ and $x_4$ (as $x_4'$ cannot be adjacent to $x_3'$). 
Now we know that all vertices in $(X_1 \cup X_2)\setminus \{x_1,x_2\}$ are adjacent to both ends of the edge $x_3'x_4$ and so there are no edges in $(X_1 \cup X_2)\setminus \{x_1,x_2\}$.
Furthermore $x_1x_2 \notin E(G)$ as this would create a $K_4$ with $x_4$ and any vertex in $X_3 \setminus \{x_3,x_3'\}$.
If all vertices in $X_3 \setminus \{x_3,x_3'\}$ were adjacent to all of $X_4 \setminus \{x_4,x_4'\}$ there would be at least $(n-2)^2$ edges in $G$ which is more than $18n$ edges for $n \geq 22$.
Therefore we can assume there is a vertex $v_3 \in X_3 \setminus \{x_3,x_3'\}$ and a vertex $v_4 \in X_4 \setminus \{x_4,x_4'\}$ which is not adjacent to $v_3$.
Then $v_4$ must be adjacent to both ends of an edge $e$ in $N(v_3)\cap (X_1 \cup X_2)$.
This edge has exactly one end in $\{x_1,x_2\}$.
If the edge $e$ is incident to $x_2$ but not $x_1$ then it forms a $K_4$ with $v_3$ and $x_4$.
If instead $e$ is incident to $x_1$ but not $x_2$ it forms a $K_4$ with $x_3'$ and $v_4$.

It follows from Lemma~\ref{Lem:Deg4Vertices} and the above that there can be at most two parts with minimum degree exactly $4$ otherwise we would have a degree $4$ vertex with just one neighbour in the part containing another degree $4$ vertex.
\end{proof}

Another distinctive feature of the upper bound construction is that low degree vertices are not adjacent to other low degree vertices.
In proving the lower bound it is helpful to prove that at most a constant number of low degree vertices are adjacent to other low degree vertices.
We do that in the following lemma.

\begin{lemma}\label{Lem:AdjLowDegVertices}
For any $k \geq 5$ suppose $G$ is a $(K_4,K_4[n])$-saturated graph on $X_1 \cup \cdots \cup X_4$.  Then there are at most $24k^2(2k^2)^{2k^2}$ vertices $v$ such that $5 \leq \deg(v) \leq k$ and $v$ is adjacent to another vertex of degree between $5$ and $k$.
\end{lemma}
\begin{proof} 
Call a vertex \emph{bad} if it satisfies $5 \leq \deg(v) \leq k$ and is adjacent to another vertex with degree between $5$ and $k$. 
Let $K=24k^2(2k^2)^{2k^2}$ and suppose for contradiction that there are more than $K$ bad vertices in $G$.
Without loss of generality assume there are at least $\frac{K}{4}$ such vertices in $X_1$.
Call the set of these vertices $A_0$ and let $B_0$ denote the set of bad vertices in $X_2 \cup X_3 \cup X_4$ which are adjacent to a bad vertex in $A_0$.
By counting $e(A_0,B_0)$ from each side we see that $|A_0| \leq e(A_0,B_0) \leq k|B_0|$ and hence $|B_0| \geq \frac{K}{4k}$.
By averaging we may assume without loss of generality that there are at least $\frac{K}{12k}$ bad vertices in $X_2$ adjacent to vertices in $A_0$.
Let $B_1$ denote $B_0 \cap X_2$ and let $A_1$ be the vertices of $A_0$ which have a neighbour in $B_1$.
Then every vertex in $A_1$ and $B_1$ has a neighbour in the other.
By double counting we see that $|B_1|\leq e(A_1,B_1) \leq k|A_1|$ and so we know that both $A_1$ and $B_1$ contain at least $\frac{K}{12k^2}$ vertices.
\medskip

For $i=0,...,k^2 +1$ we construct a collection of sets $U_i \subset X_1$, $V_i \subset X_2$ such that $U_{i+1}\subset U_i$ and $V_{i+1}\subset U_i$.  We also select vertices $u_i \in U_i$ and edges $e_i \in E(X_3,X_4)$ such that the following properties are satisfied for all $i=0,...,k^2 +1$.
\begin{enumerate}[label=(\roman*) ]\label{Enu:Properties}
\item\label{Properties:VtoE} All vertices in $V_{i+1}$ are adjacent to both endpoints of $e_{i+1}$.

\item\label{Properties:utoE} The vertex $u_i$ is adjacent to both endpoints of $e_{i+1}$.

\item\label{Properties:Vsize} $|V_i| \geq \frac{K}{12k} (2k^2)^{-i}=2k(2k^2)^{2k^2 -i}$.

\item\label{Properties:UAdjV} Each vertex in $U_i$ has a neighbour in $V_i$.

\item\label{Properties:VAdjU} Each vertex in $V_i$ has a neighbour in $U_i$.

\item\label{Properties:Usize} $|U_i| \geq \frac{K}{12k^2} (2k^2)^{-i}=2(2k^2)^{2k^2 -i}$.

\end{enumerate}
Before constructing these objects we show how they prove the lemma.
Since $|V_i|\geq 2k(2k^2)^{2k^2 -i}$ we see that the set $V_{k^2 +1}$ is non-empty.
Any vertex in $V_{k^2 +1}$ is adjacent to both ends of all the edges $e_1,...,e_{k^2}$.
As vertices in $V_{k^2 +1}$ have at most $k$ neighbours it must be that two of these edges are the same.
If $e_s=e_t$ for some $s<t \leq k^2$ then we have that $u_t$ is adjacent to some vertex $v$ in $V_s$.
As $v$ is in $V_s$ it is adjacent to both ends of $e_s$ and so forms a $K_4$ along with $u_t$.
This gives our contradiction.

We begin constructing these objects by letting $U_0=A_1$ and $V_0=B_1$.
Given $U_i$ and $V_i$ satisfying the above properties we choose any $u_i \in U_i$ and will find $U_{i+1}$, $V_{i+1}$, $e_i$ and $P_i$ satisfying the properties above.
By saturation for any vertex $v$ in $V_i \setminus N(u_i)$ there exists an edge $e \in E(X_3,X_4)$ such that both $v$ and $u_i$ are adjacent to both of the endpoints of $e$.
Since $u_i$ has at most $k$ neighbours there are fewer than $k^2$ such candidates for $e$ and hence at least $\frac{1}{k^2}|V_i \setminus N(u_i)|$ vertices of $V_i \setminus N(u_i)$ are adjacent to the endpoints of the same edge $e \in E(X_3,X_4)$.
Let $e_{i+1}$ be this edge and let $V_{i+1}$ be the vertices of $V_i \setminus N(u_i)$ that are adjacent to both ends of $e_{i+1}$.
From this we see that properties~\ref{Properties:VtoE} and~\ref{Properties:utoE} hold.

Using $|V_i| \geq \frac{K}{12k} (2k^2)^{-i} \geq 2k$ we then have
\begin{equation*}
\begin{split}
|V_{i+1}|  \geq &\frac{1}{k^2}|V_i \setminus N(u_i)| \geq \frac{1}{k^2}(|V_i|-k)\\ \geq &\frac{1}{2k^2}|V_i| \geq \frac{K}{12k} (2k^2)^{-(i+1)}\,.
\end{split}
\end{equation*}
This gives property~\ref{Properties:Vsize}.
We let $U_{i+1}= U_i \cap N(V_{i+1})$ which ensures \ref{Properties:UAdjV} and \ref{Properties:VAdjU}.
Therefore $|V_{i+1}| \leq e(U_{i+1},V_{i+1}) \leq k |U_{i+1}|$ and we see that $|U_{i+1}| \geq \frac{1}{k}|V_{i+1}| \geq \frac{K}{12k^2} (2k^2)^{-(i+1)}$ giving \ref{Properties:Usize}.
\end{proof}

With these lemmas we are now ready to prove Theorem~\ref{Thm:K4Sat}.
\begin{proof}[Proof of Theorem~\ref{Thm:K4Sat}]
Let $G$ be a $(K_4,K_4[n])$-saturated graph.

We first make the following claim, the proof of which we postpone, about the minimum degree conditions of the parts of $G$.
\begin{claim}\label{Cl:ExactDegParts}
If $G$ has at most $18n-21$ edges then $G$ has precisely two parts of minimum degree exactly $4$ and two parts of minimum degree exactly $5$.
\end{claim}
From Lemma~\ref{Lem:Deg4Parts} we know that all degree $4$ vertices in the two minimum degree $4$ parts have two neighbours in the other minimum degree $4$ part.
We can now assume we have degree $4$ vertices $a_1 \in X_1$ and $a_3 \in X_3$.  
Let the neighbours of $a_1$ be $x_2$, $x_3$, $x_3'$ and $x_4$.
We see that all vertices in $X_3 \setminus \{x_3,x_3'\}$ (including $a_3$) are adjacent to $x_2$ and $x_4$ and that $x_2$ and $x_4$ are adjacent.
Let the other two neighbours of $a_3$ be $x_1$ and $x_1'$.
Since any vertex $v$ in $X_2 \setminus x_2$ is not adjacent to $a_1$, adding the edge $a_1v$ must create a $K_4$ using $v$ and $a_1$.
Similarly, since any vertex $v$ in $X_2 \setminus x_2$ is not adjacent to $a_3$, adding the edge $a_3v$ must create a $K_4$ using $v$ and $a_3$.
This implies that $v$ is adjacent to $x_4$ and that $x_4$ is adjacent to one of $x_1$ or $x_1'$ and also one of $x_3$ or $x_3'$.
Without loss of generality assume we have the edges $x_1'x_4$ and $x_3'x_4$.
Similar arguments with a vertex in $X_4 \setminus x_4$ show that all vertices in $X_4$ are adjacent to $x_2$ and also that we have the edges $x_1x_2$ and $x_2x_3$.

We further see that by saturation every vertex of $(X_1\cup X_3)\setminus \{x_1,x_1',x_3,x_3'\}$ is adjacent to $x_2$ and $x_4$.
This means there are no edges with both ends lying in $(X_1\cup X_3)\setminus \{x_1,x_1',x_3,x_3'\}$.
All vertices in $X_2 \setminus x_2$ are adjacent to $x_1'$, $x_3'$ and $x_4$.
All vertices of $X_4 \setminus x_4$ are adjacent to $x_1$, $x_3$ and $x_2$.
\medskip

We now have that all vertices in $(X_1 \cup X_3)\setminus \{x_1,x_1',x_3,x_3'\}$ are adjacent to $x_2$ and $x_4$.  
All vertices in $X_2 \setminus x_2$ are adjacent to $x_1'$, $x_3'$ and $x_4$ whilst all vertices in $X_4 \setminus x_4$ are adjacent to all of $x_1$, $x_2$ and $x_3$.

The following claim, for which we again postpone the proof, gives us more conditions on the neighbourhoods of various vertices.
\begin{claim}\label{Cl:MoreDetails}
All vertices in $X_1 \setminus \{x_1,x_1'\}$ are adjacent to $x_3$ and $x_3'$.
All vertices in $X_3 \setminus \{x_3,x_3'\}$ are adjacent to $x_1$ and $x_1'$.
All vertices in $(X_2 \cup X_4) \setminus \{x_2,x_4\}$ are adjacent to at least 3 of $\{x_1,x_1',x_3,x_3'\}$.
Both $x_1x_3'$ and $x_1'x_3$ are edges of $G$.
\end{claim}

Under the assumption of Claim~\ref{Cl:MoreDetails} we now see that all vertices in $X_2 \setminus x_2$ are adjacent to $x_1'$, $x_3'$, $x_4$ and one of $x_1$ or $x_3$.
Let $A^1$ denote the set of vertices in $X_2 \setminus x_2$ which are adjacent to $x_1$ but not $x_3$ and let $A^3$ denote the set of vertices in $X_2 \setminus x_2$ which are adjacent to $x_3$ but not $x_1$.

Similarly all vertices in $X_4 \setminus x_4$ are adjacent to $x_1$, $x_3$, $x_2$ and one of $x_1'$ or $x_3'$.
Let $B^1$ denote the set of vertices in $X_4 \setminus x_4$ which are adjacent to $x_1'$ but not $x_3'$ and let $B^3$ denote the set of vertices in $X_4 \setminus x_4$ which are adjacent to $x_3'$ but not $x_1'$.

Adding any edge between $A^1$ and $B^1$ (likewise between $A^3$ and $B^3$) cannot create a $K_4$ so by saturation the induced graphs on $(A^1,B^1)$ and $(A^3,B^3)$ are complete.
Any edge between $A^1$ and $B^3$ would create a $K_4$ with $x_1x_3'$ whilst any edge between $A^3$ and $B^1$ would give a $K_4$ using $x_1'x_3$ therefore the bipartite graphs on $(A^1,B^3)$ and $(A^3,B^1)$ are empty.

Hence we see that there are at least 
\begin{equation*}
\begin{split}
&5\big(2n-2-|A^1|-|A^3|-|B^1|-|B^3|\big) \\ +&4\big(|A^1|+|A^3|+|B^1|+|B^3|\big) \\ 
+&|A^1||B^1|+|A^3||B^3|+4n-4+1
\end{split}
\end{equation*}
edges with at least one end in $X_2 \cup X_4$.
The $+1$ term comes from the edge $x_2x_4$ and the $+4n-4$ term comes from the edges with one end in $\{x_2,x_4\}$ and the other end in $X_1 \cup X_3$.
Along with the $4n-6$ edges between $X_1$ and $X_3$ this gives a total of at least 
\begin{equation}\label{eq:final}
18n-21 +\big(|A^1|-1\big)\big(|B^1|-1\big)+\big(|A^3|-1\big)\big(|B^3|-1\big)
\end{equation}
edges. 
We argue that either $A^1$ or $B^1$ being non-empty implies the other is non-empty.

Suppose there were a vertex in $A^1$.  
Then because it has degree at least $5$ but is not adjacent to $x_3$ it has a neighbour $v$ in $X_4 \setminus x_4$.  
This neighbour $v$ cannot be adjacent to $x_3'$ or we would have a $K_4$.  
Therefore $v \in B^1$.
Similarly for a vertex in $B^1$.
Likewise either of $A^3$ or $B^3$ being non-empty implies the other is also non-empty. 

This now means we have at least $18n-21$ edges.
Furthermore, since $A^1 \cup A^3 = X_2 \setminus x_2$ and $B^1 \cup B^3 = X_4 \setminus x_4$, equality in \eqref{eq:final} is attained only if either $|A^1|=|B^3|=1$ or $|A^3|=|B^1|=1$.
Letting $x_2'$ and $x_4'$ be the vertices in the sets of size $1$ we have our extremal construction.

It remains to prove Claims~\ref{Cl:ExactDegParts} and~\ref{Cl:MoreDetails}.

\begin{claimproof}[Proof of Claim~\ref{Cl:ExactDegParts}]
We use Lemma~\ref{Lem:AdjLowDegVertices} applied with $k=180$. As in Lemma~\ref{Lem:AdjLowDegVertices} we refer to vertices of degree between $5$ and $k$ which are adjacent to another such vertex as \emph{bad}.

We now split our vertices into groups by their degrees and whether or not they are bad, and then count edges of $G$ by counting edges between these groups.

We label our groups as follows
\begin{itemize}
\item $\vbad$ is the set of bad vertices.
\item $A:= \{v: \deg(v) \geq k+1\}$.
\item $B:= \{v: 5 \leq \deg(v)\leq k\}\setminus \vbad$.
\item $C:= \{v: \deg(v)=4\}$.
\end{itemize}

We note that vertices in $B\cup C$ only have neighbours in $A$.

Now $e(G) \geq e(B,A)+e(C,A)\geq 5|B|+4|C|$.
We also have $e(G) \geq e(A,V(G)) \geq \frac{k+1}{2}|A|$.
If $|A|\geq \frac{36n}{k+1}$ this gives at least $18n$ edges so we may assume $|A|< \frac{36n}{k+1}$.

Along with the fact that $|\vbad| \leq K=24k^2(2k^2)^{(2k^2)}$ we see that $|B|\geq 4n-|C|- K-\frac{36n}{k+1}$.
Since $e(G) \geq 5|B|+4|C|$ we have at least $20n-|C|-5K-\frac{180n}{k+1}$ edges.

If we have at most one $X_i$ with minimum degree $4$ we know $|C| \leq n$.
This implies that $G$ has at least $19n-5K-\frac{180n}{k+1}$ edges.
For $k=180$ and large enough $n$ this is at least $18n$.

We can also rule out the possibility of there being a part with minimum degree greater than $5$.
With $\vbad$, $A$, and $C$ defined as above let $B^{(5)}:=\{v \in B: \deg(v)=5\}$ and let $B^{(6+)}:=\{v \in B: \deg(v)\geq 6\}$.
We still have that $|B|=|B^{(5)}|+|B^{(6+)}|\geq 4n-|C|- K-\frac{36n}{k+1}$ and $|C| \leq 2n$.
If one part had minimum degree at least $6$ that would imply that $|B^{(5)}|\leq n$ and so we would have 
\begin{equation*}
\begin{split}
e(G) \geq & 6|B^{(6+)}|+5|B^{(5)}|+4|C| \\
= & 6|B| - |B^{(5)}| + 4|C| \\
\geq & 6\big(4n -|C| - K - \tfrac{36n}{k+1}\big) - |B^{(5)}| + 4|C| \\
= & 24n - 2|C| - |B^{(5)}| - 6K - \tfrac{216n}{k+1} \\
\geq & 19n - 6K - \tfrac{216n}{k+1}\,.
\end{split}
\end{equation*}
For $k=216$ and $n$ large enough this is more than $18n$.
\end{claimproof}
\begin{claimproof}[Proof of Claim~\ref{Cl:MoreDetails}]
We first consider a degree $5$ vertex, $a_2$, in $X_2 \setminus x_2$.
We consider separately the cases of whether $a_2$ is adjacent to neither, one, or both of $x_1$ and $x_3$.
\medskip

Firstly we suppose the vertex $a_2$ is not adjacent either of $x_1$ or $x_3$.
Adding the edge $a_2x_1$ must create a $K_4$ using $a_2$, $x_1$ and a vertex in $X_4$.  
Since $x_4$ is not adjacent to $x_1$ it must be the case that $a_2$ has a neighbour $x_4' \in X_4 \setminus x_4$.
If $a_2$ had two no neighbours in $(X_1 \cup X_3) \setminus \{x_1',x_3'\}$ there would have to be an edge from $x_1'$ to $x_3'$ but this would create a $K_4$.
Assume, without loss of generality, that $a_2$ has a neighbour $x_1'' \in X_1 \setminus \{x_1,x_1'\}$.
By considering vertices in $X_4 \setminus N(a_2)$ we see that $x_3'$ is adjacent to $x_1''$.
This means we now have a $K_4$ on the vertices $x_1'',a_2,x_3',x_4$.
\medskip

If instead $a_2$ had exactly one neighbour from $\{x_1,x_3\}$ then by symmetry we may assume it is adjacent to $x_1$ but not $x_3$.
By saturation the addition of the edge $a_2x_3$ must create a $K_4$.
Since $x_3$ is not adjacent to $x_4$ the vertex $a_2$ must have a neighbour $x_4'$ in $X_4 \setminus x_4$.
Now $a_2$ is adjacent to $x_1$, $x_1'$, $x_3'$, $x_4$ and $x_4'$ and because $a_2$ has degree 5 these are all of its neighbours.
As the only neighbour of $a_2$ in $X_3$ is $x_3'$ it must be the case that all vertices in $(X_1 \cup X_4)\setminus N(a_2)$ are adjacent to $x_3'$.
We also see that if any vertex $v$ in $X_3 \setminus \{x_3,x_3\}$ were not adjacent to $x_1'$ then, since adding the edge $a_2v$ must create a $K_4$, we must have that $v$ is adjacent to $x_1$ and $x_4'$ which would create a $K_4$ on $\{x_1,x_2,v,x_4'\}$.
Therefore every vertex in $X_3 \setminus \{x_3,x_3'\}$ is adjacent to $x_1'$.
By considering vertices on $X_4 \setminus N(a_2)$ it must also be the case that $x_3'$ is adjacent to $x_1$.
From the fact that $x_3$ is not adjacent to $a_2$ we can see that $x_3$ must be adjacent to $x_1'$ and that $x_4'$ is also adjacent to $x_1'$.
Now consider a degree 5 vertex, $a_4$ in $X_4 \setminus \{x_4,x_4'\}$.
We know that $a_4$ is adjacent to $x_3'$ and we split into the case of when $a_4$ is adjacent to $x_1'$ or not.
\medskip

If $a_4$ is not adjacent to $x_1'$ then $a_4$ has a neighbour $x_2' \in X_2 \setminus x_2$.
We know that $x_2'$ is adjacent to $x_1'$.
In order to create a $K_4$ if $a_4x_1'$ were added it must be the case that $x_2'$ is adjacent to $x_3$.
As $x_1$ is the only neighbour of $a_4$ in $X_1$ is must be the case that all vertices in $(X_2 \cup X_3) \setminus N(a_4)$ are adjacent to $x_1$.
Now all vertices in $(X_3 \cup X_4)\setminus \{x_3,x_3',x_4\}$ are adjacent to both $x_1$ and $x_2$ which are themselves adjacent to each other.
Therefore there are no edges between $X_3\setminus \{x_3,x_3'\}$ and $X_4 \setminus x_4$.
We also know that all vertices in $(X_2 \cup X_4)\setminus \{x_2,x_2',x_4,x_4'\}$ are adjacent to both ends of the edge $x_1x_3'$.
Hence there are no edges between $X_2 \setminus \{x_2,x_2'\}$ and $X_4 \setminus \{x_4,x_4'\}$.
Since all vertices in $(X_1 \cup X_2)\setminus \{x_1,x_1',x_2\}$ are adjacent to $x_3'$ and $x_4$ there are no edges between $X_1 \setminus \{x_1,x_1'\}$ and $X_2 \setminus x_2$.
In particular any vertex $v$ in $X_1 \setminus \{x_1,x_1'\}$ is not adjacent to $x_2'$ and by considering the $K_4$ created if $a_4v$ were added we see that $v$ is adjacent to $x_3$.
Since $v$ was arbitrary all vertices in $X_1 \setminus \{x_1,x_1'\}$ are adjacent to $x_3$.
This proves the lemma for this case.
\medskip

If instead $a_4$ is adjacent to $x_1'$ then as $a_4$ is of degree 5 and is adjacent to $x_1,x_1',x_2,x_3$, and $x_3'$ these are all of its neighbours.
Any vertex in $X_1 \setminus \{x_1,x_1'\}$ is non-adjacent to $a_4$ and so must be adjacent to both ends of some edge in $N(a_4)$.
This edge must be $x_2x_3$ and so all vertices in $X_1 \setminus \{x_1,x_1'\}$ are adjacent to $x_3$.
Similarly vertices in $X_3 \setminus \{x_3,x_3'\}$ are non-adjacent to $a_4$ and so must be adjacent to $x_1$.
All vertices in $X_2 \setminus x_2$ are non-adjacent to $a_4$ and hence must be adjacent to an edge in $N(a_4)$ implying each vertex in $X_2 \setminus x_2$ is adjacent to at least one of $x_1$ or $x_3$.
\medskip

Finally we consider the case where $a_2$ is adjacent to both $x_1$ and $x_3$.
We can assume all degree 5 vertices in $X_4$ are adjacent to both $x_1'$ and $x_3'$ or we would be in a situation symmetric to the last case we considered.
Let $a_4$ be such a degree 5 vertex in $X_4$.
Since all vertices in $X_1 \setminus \{x_1,x_1'\}$ and $X_3 \setminus \{x_3,x_3'\}$ are not adjacent to either $a_2$ or $a_4$ they must be adjacent to both ends of an edge in $N(a_2)$ and both ends of an edge in $N(a_4)$.  
This implies that vertices in $X_1 \setminus \{x_1,x_1'\}$ are adjacent to $x_3$ and $x_3'$ and that vertices in $X_3 \setminus \{x_3,x_3'\}$ are adjacent to $x_1$ and $x_1'$.
Similarly we see that vertices in $X_4 \setminus x_4$ are non-adjacent to $a_2$ and hence must be adjacent to an edge in $N(a_2)$.
Therefore all vertices in $X_4 \setminus x_4$ are adjacent to one of $x_1'$ or $x_3'$.
Similarly all vertices in $X_2 \setminus x_2$ are adjacent to one of $x_1$ or $x_3$.
This also shows that at least one of the edges $x_1x_3'$ or $x_1'x_3$ exists.
If one of them is not present, say $x_1x_3' \notin E(G)$ then by saturation there is some adjacent pair $b_2\in X_2\setminus x_2$, $b_4\in X_4 \setminus x_4$ which are both adjacent to $x_1$ and $x_3'$.
We also know, however, that $b_2$ and $b_4$ are both adjacent to $x_1'$ and $x_3$ but this gives a $K_4$ on $x_1',b_2,x_3,b_4$.
Therefore both $x_1x_3'$ and $x_1'x_3$ exist.
\end{claimproof}
This completes the proof.
\end{proof}

\section{Saturation Numbers of Paths and Stars}\label{Section:PathsStars}

We begin this section by determining the partite saturation numbers of stars on at least three vertices.
\begin{lemma}\label{Lem:CutVertex}
For any $r \geq 2$, $n \in \NN$ and any connected graph $H$ which contains a vertex $v$ such that $H\setminus v$ has $r$ components we have $\satp(H,H[n]) \geq (r-1)n^2$.
\end{lemma}
\begin{theorem}\label{Thm:StarSat}
For any $r \geq 2$ and $n \in \NN$ all $(K_{1,r},K_{1,r}[n])$-partite-saturated graphs have exactly $(r-1)n^2$ edges.
\end{theorem}
We show how Theorem~\ref{Thm:StarSat} follows from Lemma~\ref{Lem:CutVertex} before proving Lemma~\ref{Lem:CutVertex} itself.
\begin{proof}[Proof of Theorem~\ref{Thm:StarSat}]
The star $K_{1,r}$ has a vertex $v$ such that $K_{1,r} \setminus v$ has $r$ connected components and hence $\satp(K_{1,r},K_{1,r}[n]) \geq (r-1)n^2$.  
For any $(K_{1,r},K_{1,r}[n])$-partite-saturated graph $G$ any vertex in the part corresponding to the centre of the star must have degree at most $(r-1)n$ or by the pigeonhole principle it would have a neighbour in each remaining part giving a partite copy of $K_{1,r}$.  
This maximum degree condition implies at most $(r-1)n^2$ edges.
\end{proof}
\begin{proof}[Proof of Lemma~\ref{Lem:CutVertex}]
Let $v_1$ be the cut-vertex of $H$ and let $v_2,\ldots,v_{r+1}$ be neighbours of $v_1$ which are in distinct components of $H\setminus \{v_1\}$.
Let $X_i$ denote the part of $H[n]$ corresponding to $v_i$ and let $H_i$ denote the component of $v_i$ in $H \setminus \{v_1\}$.
Consider a $(K_{1,r},K_{1,r}[n])$-partite-saturated graph $G$ and an arbitrary vertex $x_1 \in X_1$.
If $x_1$ has fewer than $(r-1)n$ neighbours then there are two parts, say $X_2$ and $X_3$, such that each has a vertex non-adjacent to $x_1$.  
Call these vertices $x_2$ and $x_3$.
Since $G$ is saturated adding the edge $x_1 x_2$ must create a copy of $H$ using $x_1$ and hence there must be a copy of $H \setminus H_2$ in $G$ using $x_1$.
Similarly adding the edge $x_1 x_3$ must create a copy of $H$ implying the existence of a copy of $H \setminus H_3$ at $x_1$.
The union of these two subgraphs contains a partite copy of $H$ which contradicts $G$ being $H$-free.
Hence each vertex in $X_1$ has at least $(r-1)n$ neighbours and so $G$ has at least $(r-1)n^2$ edges.
\end{proof}

We now determine the partite saturation numbers of paths on at least $4$ vertices.
\begin{theorem}\label{Thm:PathSat}
For any $r \geq 4$ and $n \geq 2r$ we have the following.
\begin{equation}\label{Eq:PathSatThm}
\satp(P_r,P_r[n])= \begin{cases} (\frac{r}{2}-1)n^2 +(r-2)n + 3-r \text{, for $r$ even} \\ (\frac{r}{2}- \frac{1}{2})n^2 +(r-4)n +5-r \text{, for $r$ odd} \end{cases}
\end{equation}
\end{theorem}
\begin{proof}
Let $X_1,\ldots,X_r$ be the parts of $P_r[n]$ with $X_i$ adjacent to $X_{i+1}$ for each $i$.

We first give an upper bound construction.
Given subsets $A_i \subset X_i$ define the graph $G$ on $\bigcup_i X_i$ to be the graph with precisely the edges that lie in $(A_i,A_{i+1})$ or $(X_i \setminus A_i,X_{i+1})$ for some $i \leq r-1$.
For the upper bound if $r$ is even consider the graph $G$ created as above with $A_1:=X_1$, $A_r:= \emptyset$, $|A_i|=1$ for all even $i\leq r-2$ and $|A_i|=n-1$ for all odd $3\leq i\leq r-1$.
If $r$ is odd consider the construction $G$ given as above but with the $A_i$ satisfying $A_1:=X_1$, $A_r= \emptyset$, $|A_{r-1}|=n-1$, $|A_i|=1$ for all even $i\leq r-3$ and $|A_i|=n-1$ for all odd $3\leq i\leq r-2$.

For the lower bound we assume that for some $r\geq 4$ and some $n \geq 2r$ equation~\eqref{Eq:PathSatThm} does not hold. 
Then consider the least such $r$ and some $n \geq 2r$ for which \eqref{Eq:PathSatThm} fails.
In particular by this minimality and Theorem~\ref{Thm:StarSat} (which gives the partite saturation of $K_{1,2}=P_3$) we see that
\begin{equation}\label{Eq:PathSatInd}
\satp(P_{r-1},P_{r-1}[n])\geq \big(\tfrac{r-1}{2}-1\big)n^2 \,.
\end{equation} 
Now consider a $(P_r,P_r[n])$-partite-saturated graph $G$ on $X_1\cup \cdots \cup X_r$.
Let $N_2$ denote the set of vertices in $X_2$ which are adjacent to at least one vertex of $X_1$.
For each $i\geq 3$ let $N_i$ denote the set of vertices of $X_i$ which are adjacent to at least one vertex of $N_{i-1}$.
Since there can be no partite path on $r$ vertices it must be the case that $N_r = \emptyset$.
If $N_{r-1} = \emptyset$ then $(X_{r-1},X_r)$ must be complete in $G$ as adding an edge to this pair cannot create a partite copy of $P_r$.  
If $(X_{r-1},X_r)$ is complete then $X_1 \cup \cdots \cup X_{r-1}$ is $(P_{r-1},P_{r-1}[n])$-partite-saturated so by \eqref{Eq:PathSatInd} there are at least $\frac{r}{2}n^2$ edges in $G$.  This is at least as many as required.
Therefore we may assume $N_i \neq \emptyset$ for all $2 \leq i \leq r-1$.
If $N_i = X_i$ for some $i\geq 2$ then the pairs $(X_j,X_{j+1})$ are complete for all $1 \leq j \leq i-1$.
Then $X_i \cup \cdots \cup X_r$ is $(P_{r-i+1},P_{r-i+1}[n])$-partite-saturated so by \eqref{Eq:PathSatInd} there are at least $(\frac{r-1}{2})n^2$ edges in $G$.  
This is at least as many as required.
We now assume $N_i \neq X_i$ for all $2 \leq i \leq r$ so for all $i=2,\ldots,r-1$ we have $1 \leq |N_i| \leq n-1$.
For each $i \geq 2$ let $\overline{N_i}$ denote $X_i \setminus N_i$.
We observe that $(X_1,N_1)$ and $(\overline{N_{r-1}},X_r)$ must be complete.
As are $(N_i,N_{i+1})$ and $(\overline{N_i} ,X_{i+1})$ for $2 \leq i \leq r-2$ because adding edges to either of these pairs cannot create a partite copy of $P_r$.
Therefore we find that $G$ has all possible edges except those in pairs $(X_1,\overline{N_{2}})$ or $(N_i, \overline{N_{i+1}})$ for $2 \leq i \leq r-1$ and so $e(G)$ is at least
\begin{equation}\label{Eq:PathSat1}
(r-1)n^2 - n|\overline{N_{2}}|- \sum_{i=2}^{r-1} |N_i||\overline{N_{i+1}}|=(r-2)n^2 + n|N_2| -n\sum_{i=2}^{r-1}|N_i| +\sum_{i=2}^{r-2}|N_i||N_{i+1}|\,.
\end{equation}
Suppose $N_2,...,N_{r-1}$ have been chosen to minimise the above expression under the assumption that each $|N_i|$ is between $1$ and $n-1$.
The contribution to \eqref{Eq:PathSat1} from terms that include $N_2$ is exactly $|N_2||N_3|$ which (regardless of the value of $|N_3|$) is minimised by taking $|N_2|=1$.
For $3 \leq i \leq r-2$ the contribution to \eqref{Eq:PathSat1} from terms that include $N_i$ is
\begin{equation*}
|N_i|\big(|N_{i-1}|+|N_{i+1}|-n\big)\,.
\end{equation*}
When $|N_{i-1}|=1$ the above expression is at most zero and so minimised by taking $|N_i|=n-1$.
If $|N_{i-1}|=n-1$ it is at least zero and so minimised by taking $|N_i|=1$.
In this way using $|N_2|=1$ we can see that for $2 \leq i \leq r-2$ we have the following.
\begin{equation*}
|N_i|= \begin{cases} 1 \text{, for $i$ even} \\ n-1 \text{, for $i$ odd} \end{cases}
\end{equation*}
The contribution to \eqref{Eq:PathSat1} from the $N_{r-1}$ terms is $|N_{r-1}|\big(|N_{r-2}|-n \big)$ which is always negative and so the expression is minimised when $|N_{r-1}|=n-1$.
The graph given with the $N_i$ taking these sizes is the same as our upper bound construction completing the proof.
\end{proof}

\section{$2$-Connectivity and the Growth of Saturation Numbers}\label{Section:2-connectivity}

Recall that a graph is \emph{$2$-connected} if after the removal of any single vertex it is still connected.
Observe that if $H'$ can be obtained from $H$ by adding or removing isolated vertices then $\satp(H,H[n])=\satp(H',H'[n])$.
It is also clear that $\satp(K_2,K_2[n])=0$.
\begin{theorem}\label{Thm:2-connectivity}
For any graph $H$ with $e(H) \geq 2$ and no isolated vertices, if $H$ is $2$-connected then $\satp(H,H[n])= \Theta(n)$ and if $H$ is not $2$-connected then $\satp(H,H[n])= \Theta(n^2)$.
\end{theorem}

\begin{proof}
If $H$ is connected but not $2$-connected then there must be a cut vertex, $v$, of $H$ such that $H \setminus v$ has at least two components.  Then by Lemma~\ref{Lem:CutVertex} we have $\satp(H,H[n]) \geq n^2$.
\smallskip

We now consider the case when $H$ is disconnected but has no isolated vertices. 
Let $H_1$ and $H_2$ be two connected components of $H$.
$G \subset H[n]$ is $(H,H[n])$-partite-saturated then by saturation the induced graph of $G$ onto at least one of $H_1[n]$ or $H_2[n]$ must be complete.
Since each $H_i$ contains an edge this means $G$ has at least $n^2$ edges.
\smallskip

Finally we consider the case when $H$ is $2$-connected.  
The fact that $\satp(H,H[n])= \Omega(n)$ comes from the fact that in an $(H,H[n])$-saturated graph $G$ every vertex, $x$, has degree at least one.  If not adding an edge incident to $x$ would not create a copy of $H$ since $H$ has minimum degree at least two by 2-connectivity.

We now give an upper bound construction.  
For each edge $ij$ of $H$ we define $H_{ij}$ to be the graph obtained from $H$ be removing all edges incident to $i$ or $j$ including the edge $ij$.  
We define $V_i(H_{ij})$ to be the vertices of $H_{ij} \setminus \{i,j\}$ which were incident to $i$ in $H$.
Similarly $V_j(H_{ij})$.  
For $n \geq e(H)$ we let $G_1 \subset H[n]$ be the disjoint union of a copy of $H_{ij}$ for each edge $ij$ of $H$.  
Create $G_2$ from $G_1$ by adjoining each vertex of $V_i(H_{ij})$ (in the copy of $H_{ij}$ in $G_1$) to every vertex in $X_i \setminus V(G_1)$, and by adjoining each vertex of $V_j(H_{ij})$ to every vertex in $X_j \setminus V(G_1)$ for each edge $ij$of $H$.  
We then create $G_3$ from $G_2$ by arbitrarily adding edges until the graph is $(H,H[n])$-partite-saturated.

We claim that $G_3$ is $(H,H[n])$-partite-saturated and has at most $2e(H)^2 n - e(H)^3$ edges.  
To prove this it is sufficient to show that $G_2$ has no partite copy of $H$ and that $G_3$ has at most $2e(H)^2 n - e(H)^3$ edges.
We first note that there are no edges of $G_2$ or $G_3$ with both end points in $V(G_3) \setminus V(G_1)$ since any such edge $x_ix_j$ would form a copy of $H$ with the $H_{ij}$.  
We can then bound the number of edges of $G_3$ by $E(H[n]) - E(H[n-e(H)]) = n^2 e(H) - (n- e(H))^2 e(H) =2e(H)^2 n - e(H)^3$.

Suppose now for contradiction that $G_2$ has a partite copy of $H$.  
Denote the vertices of this copy of $H$ by $x_i$ for $i=1,\ldots,|H|$.  
Since $G_1$ is $H$-free at least one of the $x_i$'s lies in $V(G_2) \setminus V(G_1)$.  
Suppose without loss of generality that $x_1 \notin V(G_1)$.  
Let $x_2$ be a neighbour of $x_1$ in the partite copy of $H$.  
Since there are no edges with both end points in $V(G_2) \setminus V(G_1)$ it must be the case that $x_2 \in V(G_1)$.
Since $x_1x_2$ is an edge of $G_2$ it must be the case that $x_2 \in V_1(H_{1i})$ for some $i$ adjacent to 1 in $H$.
Suppose $x_2 \in V_1(H_{13})$.
Then similarly $x_3 \in V_1(H_{1k})$ for some $k \neq 3$.
Therefore $x_2$ and $x_3$ are in different $H_{ij}$'s and hence different connected components of $G_1$.
Our copy of $H$ is separated by following the set 
\begin{equation*}
\{x_i: x_i \notin V(G_1)\text{ and $x_i$ is adjacent to a vertex in $H_{13}$}\}\,.
\end{equation*}
Since $H$ is $2$-connected this set must contain at least two vertices, one of which is $x_1$.
The only $x_i$'s that vertices in $H_{13}$ can be adjacent to outside of $H_{13}$ are $x_1$ and $x_3$ but $x_3 \in V(G_1)$ which gives a contradiction.
\end{proof}

\section{Extra-Saturation Numbers}\label{Section:ExSat}

In this section we determine the partite extra-saturation numbers of cliques and trees, and show that of graphs on $r$ vertices the cliques have the largest partite extra-saturation numbers.

Since it follows from the proof of $\sat(K_3,K_3[n])=6n-6$ in~\cite{FerJacPfeWen} that $\exsatp(K_3,K_3[n])=6n-6$ we look only at cliques on at least 4 vertices.
The proof of the following Theorem uses ideas from~\cite{FerJacPfeWen}.

\begin{theorem}\label{Thm:CliqueExSat}
For any integer $r\geq 4$ and all large enough $n \in \NN$ we have 
\begin{equation*}
\exsatp(K_r,K_r[n])= (2n-1)\binom{r}{2}\,.
\end{equation*}
\end{theorem}
\begin{proof}
For the upper bound consider the graph $G$ consisting of a copy of $K_r$ with each vertex of this clique adjacent to all vertices in adjacent parts of $K_r[n]$.  
For the lower bound consider a $(K_r,K_r[n])$-partite-extra-saturated graph $G$ on $X_1 \cup \cdots \cup X_r$.

For all $i=1,\ldots,r$ let $\delta_i:=\min \{d(x): x \in X_i \}$.
Since for any $i$ we have $e(G) \geq \delta_i n$ we must have $\delta_i < r^2$ or $G$ would have more than $(2n-1)\binom{r}{2}$ edges.
By the fact that any vertex which is not adjacent to some part must be incident to all vertices in the other parts we see that $\delta_i \geq r-1$ for all $i$.
\begin{claim}
All vertices of degree $r-1$ are in a $K_r$.
\end{claim}
\begin{claimproof}
If $v \in X_1$ is a vertex of degree $r-1$ it must have a neighbour in each adjacent part.
Denote these by $x_i \in X_i$ for $i=2,...,r$.
For any $y_2 \in X_2 \setminus x_2$ adding the edge $vy_2$ must create a new $K_r$.
This new clique must be on $\{v,y_2,x_3,x_4,...,x_r\}$ so $x_3,...,x_r$ must all be pairwise adjacent.
Similarly for any $y_3 \in X_3 \setminus x_3$ adding the edge $vy_3$ must create a new $K_r$ (which must be $\{v,x_2,y_3,x_4,x_5,...,x_r\}$) so $x_2,x_4,x_5,...,x_r$ must all be pairwise adjacent.
This gives a $K_r$ on $v,x_2,x_3,...,x_r$.
\end{claimproof}
Let $x_i$ be a vertex of degree $\delta_i$ for each $i$.  
For each $i$ let $Y_i := \bigcup_{j\neq i} \big(N(x_j)\cap X_i\big)$ and let $Y:=\bigcup_i Y_i=\bigcup_i N(x_i)$.
Observe that $|Y| \leq r^3$.
\begin{claim}
For all $i \neq j$, each vertex in $X_i \setminus Y_i$ has a neighbour in $Y_j$.
\end{claim}
\begin{claimproof}
Given some $i \neq j$ and a vertex $v \in X_i \setminus Y_i$ consider any $k \in \{1,...,r\}\setminus \{i,j\}$.
As $v$ is not in $Y_i$ it must be that $v$ is not adjacent to $x_k$.
Therefore, by saturation, adding $vx_k$ creates a new $K_r$.
This $K_r$ must use a neighbour of $x_k$ in $X_j$ and hence this neighbour is both in $Y_j$ and also adjacent to $v$.
\end{claimproof}
We can now lower bound the edges of $G$ by
\begin{equation}\label{eq:exsat}
\begin{split}
e(G) &\geq e(Y,X\setminus Y)+e(X \setminus Y) \\
&\geq \sum_{v \in X \setminus Y}\bigg(\deg(v,Y)+\tfrac{1}{2}\big(\deg(v,X \setminus Y)\big)\bigg) \\ 
&\geq \sum_i |X_i \setminus Y|\big(r-1 + \tfrac{1}{2}\big(\delta_i -(r-1)\big) \big) \\
&= \tfrac{1}{2}(r-1)|X \setminus Y| + \tfrac{1}{2}\sum_i |X_i \setminus Y|\delta_i \\
&\geq \tfrac{1}{2}n\bigg(r(r-1)+\sum_i \delta_i \bigg) -\tfrac{1}{2}r^3\bigg(r-1 +\sum_i \delta_i\bigg) \\
&\geq \tfrac{1}{2}n\bigg(r(r-1)+\sum_i \delta_i \bigg) -r^6 \\
&= (2n-1)\binom{r}{2} + \tfrac{1}{2}n \sum_i\big(\delta_i - (r-1) \big) + \binom{r}{2}-r^6\,.
\end{split}
\end{equation}
Therefore for $n > 2r^6$ we have $\delta_i=r-1$ for all $i$.     
Each of the $x_i$'s has one neighbour in each adjacent part and is in a copy of $K_r$.
We see that by saturation for a vertex $v$ of degree $r-1$ every vertex $w$ in a different part from $v$ which is not adjacent to $v$ is incident to all neighbours of $v$ outside of the part of $w$.
Therefore vertices of degree $r-1$ are not adjacent.
We also see that for any $i\neq j$ the vertices $x_i$ and $x_j$ have $r-2$ common neighbours and so with the sets $Y_i$ and $Y$ as before we find that $|Y_i|=1$ for all $i$, so $Y=r$.  

Using \eqref{eq:exsat} we get
\begin{equation*}
\begin{split}
e(G) &\geq e(Y, X\setminus Y) + e(X\setminus Y) + e(Y) \\
&\geq \tfrac{1}{2}(r-1)|X \setminus Y| + \tfrac{1}{2}\sum_i |X_i \setminus Y|\delta_i +e(Y) \\ 
&\geq 2(n-1)\binom{r}{2} +e(Y)\,.
\end{split}
\end{equation*}
Since there is a $K_r$ on $Y$ we have $e(Y)= \binom{r}{2}$ and the result follows.  
\end{proof}

The upper bound construction can be generalised to any $H$ by letting $G$ consist of a copy of $H$ with each vertex of this $H$ adjacent to all vertices in adjacent parts of $H[n]$.  
This gives an upper bound of 
\begin{equation*}
\exsatp(H,H[n])\leq (2n-1)e(H)\,.
\end{equation*}
In particular this shows that over graphs $H$ on $r$ vertices the cliques give rise to the largest value of $\exsatp(H,H[n])$ and also that all partite extra-saturation numbers of graphs with at least two edges are linear.

Next we determine the partite extra-saturation number of trees.

\begin{theorem}\label{Thm:TreeExSat}
For any tree $T$ on at least $3$ vertices and any natural number $n\geq 4$ we have $\exsatp(T,T[n])=(|T|-1)n$
\end{theorem}
\begin{proof}
For an upper bound construction let $G$ be the union of $n$ disjoint partite copies of $T$.
\medskip

Turning our attention to the lower bound we let $L$ denote the set of leaves of $T$ and call the vertices in $C=V(T) \setminus L$ \emph{core} vertices.

Now suppose $G$ is a $(T,T[n])$-extra-saturated graph with $n\geq 4$.
Let $x$ be a vertex of $G$ lying in a part associated to a core vertex $v \in C$.
In $G$ the vertex $x$ must either have a neighbour in each adjacent part of $T[n]$ or it must be that $\deg_G(x) \geq n(\deg_T(v)-1) \geq 2\deg_T(v)$.
This is because if $x$ had no neighbour in some adjacent part it must be adjacent to all vertices in the other adjacent parts.
Since $\deg_T(v)\geq 2$ and $n \geq 4$ this means $x$ has at least $2\deg_T(v)$ neighbours.
We let $L[n]$ and $C[n]$ denote the set of vertices in $T[n]$ that lie in parts corresponding to $L$ and $C$ respectively.

We have
\begin{equation}\label{Eq:TreeExSat}
\begin{split}
e(G) =& \sum_{x \in C[n]} \big(\tfrac{1}{2}\deg_G(x,C[n])+\deg_G(x,L[n])\big) \\
= &\frac{1}{2}\sum_{x \in C[n]} \big( \deg_G(x)+\deg_G(x,L[n])\big) \,. 
\end{split}
\end{equation}
Let $x \in C[n]$ be a vertex associated in the part associated to a vertex $v \in C$.
If $x$ is adjacent to a vertex in each adjacent part then 
\begin{equation*}
\deg_G(x)+\deg_G(x,L[n])\geq \deg_T(v)+\deg_T(v,L)
\end{equation*}
otherwise we also obtain
\begin{equation*}
\deg_G(x)+\deg_G(x,L[n])\geq \deg_G(x) \geq 2\deg_T(v) \geq \deg_T(v) + \deg_T(v,L)\,.
\end{equation*}
Using these and \eqref{Eq:TreeExSat}, we see that 
\begin{equation*}
\begin{split}
e(G) &\geq \frac{n}{2}\sum_{v \in C}\big( \deg_T(v)+\deg_T(v,L)\big) \\
&= n \cdot e(T) =n\big(|T|-1\big)
\end{split}
\end{equation*}
completing the proof.
\end{proof}

\section{Concluding Remarks}\label{Section:Conclusion}
It would be very nice to be able to determine the value of $\sat(K_r,K_r[n])$ for $r\geq 5$.
Exact answers here would probably be very difficult though it may be possible to determine up to an error term of $o(n)$ or even $O(1)$.
It would be helpful to be able to determine the following value in order to make progress on this problem.
\medskip

For integers $r \geq s \geq 3$ let $m(r,s)$ denote the fewest vertices an $r$-partite graph $G$ can have such that $G$ is $K_s$-free but every set of $s-1$ parts contains a $K_{s-1}$.
\medskip

We can use $m(r,r-1)$ and $m(r-1,r-1)$ to get upper and lower bounds respectively on $\sat(K_r,K_r[n])$.
\smallskip

For the upper bound let $F\subset K_r[n]$ be a $K_{r-1}$-free graph on $m(r,r-1)$ vertices such that any $r-2$ parts contain a $K_{r-2}$.
Create a $(K_r,K_r[n])$-saturated graph $G \subset K_r[n]$ by attaching all vertices of $F$ to all vertices outside of $F$ which lie in a different part.
Then if necessary add edges between vertices of $F$ until the graph is $(K_r,K_r[n])$-saturated.
This implies that $\sat(K_r,K_r[n])$ is less than $m(r,r-1)\cdot (r-1)n$.
Using the fact that $m(4,3)=6$ this shows that $\sat(K_4,K_4[n]) \leq 18n$ which we know from Theorem~\ref{Thm:K4Sat} to be close to the correct answer.
\smallskip

For the lower bound we prove a minimum degree condition in all $(K_r,K_r[n])$-saturated graphs.
If $G$ is a $(K_r,K_r[n])$-saturated graph note that any vertex in $G$ is either adjacent to all vertices in one part of $K_r[n]$ or its neighbourhood induces an $(r-1)$-partite graph which is $K_{r-1}$-free but where there is a $K_{r-2}$ on any $r-2$ parts.
Therefore, for $n \geq m(r-1,r-1)$ we have $\delta(G)\geq m(r-1,r-1)$ and hence $\sat(K_r,K_r[n])\geq m(r-1,r-1)\cdot rn/2$.
When $r=4$ this gives the minimum degree condition of $\delta(G) \geq m(3,3)=4$.

\bibliographystyle{amsplain}
\bibliography{Bib}
\end{document}